\documentclass{amsart}

\usepackage{graphicx}

\usepackage{amsmath}
\usepackage{amssymb}
\usepackage{amsthm}
\usepackage{amsbsy}

\newtheorem{thm}{Theorem}[section]
\newtheorem{cor}[thm]{Corollary}
\newtheorem{lem}[thm]{Lemma}
\newtheorem{prop}[thm]{Proposition}

\theoremstyle{definition}
\newtheorem{defn}[thm]{Definition}
\theoremstyle{remark}
\newtheorem{rem}[thm]{Remark}
\newtheorem{ex}[thm]{Example}

\numberwithin{equation}{section}

\newcommand{\norm}[1]{\left\Vert#1\right\Vert}
\newcommand{\abs}[1]{\left\vert#1\right\vert}
\newcommand{\eps}{\varepsilon}

\newcommand{\e}{\mathrm{e}}

\newcommand{\RR}{{\mathbb R}}
\newcommand{\CC}{{\mathbb C}}
\newcommand{\NN}{{\mathbb N}}

\newcommand{\calC}{{\mathcal C}}
\newcommand{\calO}{{\mathcal O}}

\begin{document}

\title[a differential operator representation]{%
A differential operator representation of continuous homomorphisms
between the spaces of entire functions of given proximate orders}

\author[T.~Aoki]{Takashi Aoki}
\address[T.~Aoki]{Department of Mathematics,
Kindai University, Higashi--Osaka 577-8502, Japan}
\email{aoki@math.kindai.ac.jp}
\thanks{The first author is supported by JSPS KAKENHI Grant Nos. 26400126 and 18K03385}

\author[R.~Ishimura]{Ryuichi Ishimura}
\address[R.~Ishimura]{Department of Mathematics and Informatics, Faculty of Science,
Chiba University, Yayoicho, Chiba 263-8522, Japan}
\email{ishimura@math.s.chiba-u.ac.jp}

\author[Y.~Okada]{Yasunori Okada}
\address[Y.~Okada]{Institute of Management and Information Technologies,
Chiba University, Yayoicho, Chiba 263-8522, Japan}
\email{okada@math.s.chiba-u.ac.jp}
\thanks{The third author is supported by JSPS KAKENHI Grant No. 16K05170}

\subjclass[2010]{Primary 34A15, Secondary 47F99.}
\keywords{entire functions, proximate orders,
partial differential operators of infinite order}

\dedicatory{Dedicated to the memory of Carlos A. Berenstein}

\thispagestyle{empty}


\begin{abstract}
In this paper, we consider the locally convex spaces of entire
functions with growth given by proximate orders,
and study the representation as a differential operator
of a continuous homomorphism from such a space to another one.
As a corollary, we give a characterization of continuous endomorphisms
of such spaces.
\end{abstract}

\maketitle


\section{Introduction}

In the paper \cite{aiosu},
we proved with coauthors Struppa and Uchida
that any continuous endomorphism of the space
of entire functions with a given constant order $\rho>0$
is characterized as an infinite order partial differential operator
with the symbol satisfying certain growth conditions.
Such endomorphisms play a role in the study of superoscillations
(see \cite{acss1}, \cite{acss2}).
As remarked in \cite{aiosu}, we can largely generalize these results
to the case of proximate order in the sense of Valiron \cite{V},
instead of a constant order.
Such generalizations were first proved
by X. Jin \cite[Corollaries 6.5 and 6.6]{J},
where he considered in a more general framework
of formal power series of class ${\rm M}$.

In this paper, we consider continuous homomorphisms
between spaces of entire functions
given by possibly different proximate orders,
and will give their differential operator representations.
See Theorems \ref{thm:main1} and \ref{thm:main2}.
The characterization of continuous endomorphisms
of a space given by a proximate order
is then given as corollaries (Corollaries \ref{cor:main1} and \ref{cor:main2}).
By studying homomorphisms between two spaces
rather than endomorphisms of one space,
it is now much more visible
that two proximate orders of the source and the target spaces
play almost independent roles
in the growth conditions for symbols.
See the conditions
\eqref{eq:est-D-rho1-rho2-normal}
and \eqref{eq:est-D-rho1-rho2-minimal}
in Definition~\ref{def:D-rho1-rho2}.

The plan of this paper is as follows.
After recalling the notion of a proximate order for a
positive order in \S~\ref{sec:prox-orders-their},
we prepare topological properties of the spaces of entire functions
with a given proximate order in \S~\ref{sec:spac-entire-funct}.
In \S~\ref{sec:diff-oper-repr}, we give of our main results and their proofs.
Note that we use, in this paper,
the terminology ``continuous homomorphism''
instead of ``continuous linear operator'',
although the latter seems more familiar in functional analysis.

\section{Proximate orders and their properties}
\label{sec:prox-orders-their}

We recall several notions and properties concerning
the spaces of entire functions with growth order,
principally, we will follow \cite{Ll-G}.

The notion of a proximate order was introduced by G. Valiron \cite{V}.
\begin{defn}[proximate order]\label{def:proximate-order}
A differentiable function $\varrho(r) \geq 0$ defined for $r \geq 0$
is said to be a {\em proximate order}
for the {\em order} $\rho\geq 0$ if it satisfies
\begin{itemize}
\item[(i)] $\displaystyle\lim_{r \to + \infty} \varrho(r) = \rho$,
\item[(ii)] $\displaystyle\lim_{r \to + \infty} \varrho'(r) r \ln r = 0$.
\end{itemize}
\end{defn}
Note that in many literatures (for ex., \cite{V}, \cite{Lv}, \cite{Ll-G}),
a proximate order and its order share the same symbol
like a proximate order $\rho(r)$ and its order $\rho$,
and they are distinguished
by the existence of postfixed ``$(r)$'' or ``$(\cdot)$''.
However, in this paper, we use different symbols $\varrho(r)$ and $\rho$
in order for an easy distinguishability even without a postfix.

It is well-known that if $\rho>0$,
there exists a large $r_0>0$
such that the function $r^{\varrho(r)}$ is strictly increasing on $r>r_0$,
and tending to $+\infty$.
If moreover $\rho > 1$, $r^{\varrho(r) -1}$ is also strictly increasing
for large $r$, (see e.g. \cite[Chapter I, \S 12]{Lv}).
We remark that these facts follow from
\begin{equation*}
\frac{d}{d r} r^{\varrho(r)} = r^{\varrho(r)-1} (\varrho'(r) r \ln r + \varrho(r)).
\end{equation*}

Throughout this section, $\varrho(r)$ denotes a proximate order
for a positive order $\rho=\displaystyle\lim_{r\rightarrow +\infty}\varrho(r)>0$.
As is noted in \cite[p.16, Note.]{Ll-G},
we can always take another proximate order $\hat{\varrho}(r)$
which satisfies the conditions:
\begin{gather}
\label{eq:rhohat-1}
\text{there exists a constant $r_1>0$
such that $\hat{\varrho}(r)=\varrho(r)$ for any $r\geq r_1$},
\\
\label{eq:rhohat-2}
\text{$r\mapsto r^{\hat{\varrho}(r)}$ is strictly increasing on $r>0$
and maps $(0,\infty)$ onto $(0,\infty)$}.
\end{gather}
\begin{defn}[normalization of a proximate order]\label{def:normaliz-prox}
For a proximate order $\varrho(r)$ for a positive order,
a proximate order $\hat{\varrho}(r)$ satisfying \eqref{eq:rhohat-1} and
 \eqref{eq:rhohat-2} is said to be
a {\em normalization} of $\varrho(r)$.
\end{defn}
Though there are many choices of such $\hat{\varrho}(r)$,
we fix one $\hat{\varrho}(r)$ for $\varrho(r)$,
and following to \cite{Ll-G},
we denote by $r = \varphi (t)$ for $t>0$,
the inverse function of $t = r^{\hat{\varrho}(r)}$ $(r>0)$.
As in \cite[p.203]{Ll-G},
we set
\begin{equation}\label{G-rho-q}
G_q = G_{\hat{\varrho},q} := \frac{\varphi(q)^q}{(\e \rho)^{q/\rho}},
\quad\text{for $q\in\NN$},
\end{equation}
where they denote by $A_q$ instead. 
(Note that $\varphi(t)$ for $t\gg1$ and so $G_{\hat{\varphi},q}$ for
$q\gg1$ do not depend on the choice of a normalization.)
Remark that for a constant proximate order,
that is, $\varrho(r)\equiv\rho$,
we take $\hat{\varrho}(r)=\varrho(r)\equiv\rho$
and we have $\varphi(q) = q^{1/\rho}$, which implies
\begin{equation*}
G_q
=\frac{q^{q/\rho}}{(\e \rho)^{q / \rho}}
=\Bigl(\frac{q^q}{\e^q q!}\Bigr)^{1/\rho}\rho^{-q/\rho}q!^{1/\rho}
\sim(\rho^q\sqrt{2\pi q})^{-1/\rho}q!^{1/\rho}
\sim'q!^{1/\rho}.
\end{equation*}
Here we denote by $a_q\sim b_q$ and by $a_q\sim' b_q$
the relations $a_q/b_q\rightarrow 1$ and $\ln(a_q/b_q)=O(q)$
as $q\rightarrow\infty$ respectively.

We prepare some auxiliary properties of proximate orders.

\begin{lem}\label{lem:pseudo-subadditivity}
There exist constants $\kappa>0$ and $B>0$ depending only on $\hat{\varrho}$
such that
\begin{equation*}
\forall r>0,\,\forall s>0,\quad
(r+s)^{\hat{\varrho}(r+s)}
\leq \kappa(r^{\hat{\varrho}(r)}+s^{\hat{\varrho}(s)})+B.
\end{equation*}
Precisely speaking, we can choose $\kappa$ depending only
on the order $\rho=\lim_{r\rightarrow+\infty}\varrho(r)$.
\end{lem}
\begin{proof}
By \cite[p.16, Proposition 1.20]{Ll-G},
for any $\varepsilon>0$, we have $R(\varepsilon)$ such that
\begin{equation*}
1\leq\forall k\leq 2,\,\forall r\geq R(\varepsilon),\quad
(k r)^{\hat{\varrho}(k r)}
\leq(1+\varepsilon)k^{\rho}r^{\hat{\varrho}(r)}.
\end{equation*}
Since $(k r)^{\hat{\varrho}(k r)}$ is bounded on $1\leq k\leq 2$,
$0<r<R(\varepsilon)$, we have some $C_{\varepsilon}$ such that
\begin{equation*}
1\leq\forall k\leq 2,\,\forall r>0,\quad
(k r)^{\hat{\varrho}(k r)}
\leq(1+\varepsilon)k^{\rho}r^{\hat{\varrho}(r)}+C_{\varepsilon}.
\end{equation*}
Now, for any $s$ and $r$ satisfying $0<s\leq r$,
we define $k:=1+s/r\leq 2$ and get
\begin{equation*}
(r+s)^{\hat{\varrho}(r+s)}
=(kr)^{\hat{\varrho}(kr)}
\leq(1+\varepsilon)k^{\rho}r^{\hat{\varrho}(r)}+C_{\varepsilon}
\leq(1+\varepsilon)2^{\rho}(r^{\hat{\varrho}(r)}+s^{\hat{\varrho}(s)})+C_{\varepsilon}.
\end{equation*}
The inequality
$(r+s)^{\hat{\varrho}(r+s)}
\leq(1+\varepsilon)2^{\rho}(r^{\hat{\varrho}(r)}+s^{\hat{\varrho}(s)})+C_{\varepsilon}$
also holds for $0<r\leq s$ by symmetry.

To conclude the proof, it suffices to take $\kappa>2^{\rho}$ arbitrarily
and $B=C_{\eps}$ with $\varepsilon:=\kappa/2^{\rho}-1$.
\end{proof}

\begin{lem}\label{lem:Gq}
The sequence $\{G_p\}_p$ is supermultiplicative, that is,
\begin{equation*}
G_p G_q\leq G_{p+q},
\quad\text{for any $p,q\in\NN$}.
\end{equation*}
\end{lem}
\begin{proof}
We have the following inequality
\begin{equation*}
\frac{G_p G_q}{G_{p+q}}
=\frac{\varphi(p)^p}{(\e\rho)^{p/\rho}}
\cdot\frac{\varphi(q)^q}{(\e\rho)^{q/\rho}}
\cdot\frac{(\e\rho)^{(p+q)/\rho}}{\varphi(p+q)^{p+q}}
=
\Bigl(\frac{\varphi(p)}{\varphi(p+q)}\Bigr)^p
\cdot\Bigl(\frac{\varphi(q)}{\varphi(p+q)}\Bigr)^q
\leq 1,
\end{equation*}
for $p,q\in\NN$, since $\varphi(t)$ is increasing.
\end{proof}

As for $\varphi$,
by \cite[the proof of Theorem 1.23]{Ll-G}
(c.f. for example \cite[the proof of Lemma 6.2]{J}),
we give
\begin{lem}\label{VitesseDeVarphi}
For every $\delta > 0$ with 
$\displaystyle{\delta < \frac{1}{\rho}}$, there exists $T_0 > 0$ 
such that if $t \geq T_0$, we have
\begin{equation}\label{derDelogVarphi}
\left(\frac{1}{\rho} - \delta\right) \frac{d}{d t} \ln t 
< \frac{d}{d t} \ln \varphi(t) 
< \left(\frac{1}{\rho} + \delta\right) \frac{d}{d t} \ln t.
\end{equation}
\end{lem}

\begin{lem}\label{lem:y-u-t}
For $u,t,\sigma>0$, we define
\begin{equation*}
y_{\sigma}(u,t):=\ln\frac{\varphi(t)}{\varphi(u)}-\sigma\frac{t}{u}.
\end{equation*}
Then for any $\sigma'$ with $0<\sigma'<\sigma$, there exists $T_1$
such that
\begin{equation*}
y_{\sigma}(u,t)+\frac{1}{\rho}\ln(\e\rho)
\leq -\frac{1}{\rho}\ln\sigma',
\quad\text{for any $u,t\geq T_1$}.
\end{equation*}
\end{lem}

\begin{proof}
For every $\delta$ with $0<\delta<1/\rho$,
we can apply Lemma~\ref{VitesseDeVarphi} and get a constant $T_{\delta}$
such that
\begin{equation*}
\left(\frac{1}{\rho}-\delta\right)\frac{d}{dt}\ln t
<\frac{d}{d t}\ln\varphi(t)
<\left(\frac{1}{\rho}+\delta\right)\frac{d}{d t}\ln t
\end{equation*}
holds for $t\geq T_{\delta}$.

By applying the integration $\int_u^t(\cdot)d t$
to the right inequality in the case $t\geq u\geq T_{\delta}$
and to the left one in the case $u\geq t\geq T_{\delta}$,
we get
\begin{gather*}
y_{\sigma}(u,t)
\leq\Bigl(\frac{1}{\rho}+\delta\Bigr)\ln\frac{t}{u}-\sigma\frac{t}{u},
\quad\text{for $t\geq u\geq T_{\delta}$},
\\
y_{\sigma}(u,t)
\leq\Bigl(\frac{1}{\rho}-\delta\Bigr)\ln\frac{t}{u}-\sigma\frac{t}{u},
\quad\text{for $u\geq t\geq T_{\delta}$}.
\end{gather*}
Using the notations $v:=t/u$ and
\begin{equation*}
g_{\pm}(v):=\Bigl(\frac{1}{\rho}\pm\delta\Bigr)\ln v-\sigma v,
\quad\text{for $v>0$},
\end{equation*}
$y_{\sigma}(u,t)$ admits the estimate
\begin{equation*}
y_{\sigma}(u,t)\leq\max\{\sup_{v\geq 1}g_{+}(v),\sup_{0<v\leq 1}g_{-}(v)\},
\quad\text{for $u,t\geq T_{\delta}$}.
\end{equation*}
Since
\begin{equation*}
g_{\pm}'(v) = (\frac{1}{\rho}\pm\delta) \frac{1}{v}-\sigma=0
\ \Leftrightarrow \ 
v=v_{\pm}:=\frac{\frac{1}{\rho}\pm\delta}{\sigma}
=\frac{1}{\sigma \rho}\pm\frac{\delta}{\sigma}>0,
\end{equation*}
we have
\begin{equation*}
\sup_{v>0}g_{\pm}(v)=g_{\pm}(v_{\pm})
=\Bigl(\frac{1}{\rho}\pm\delta\Bigr)\ln\frac{1/\rho\pm\delta}{\e\sigma},
\end{equation*}
which implies
\begin{align*}
y_{\sigma}(u,t)+\frac{1}{\rho}\ln(\e\rho)
&\leq
\max_{\pm}\Bigl[
\bigl(\frac{1}{\rho}\pm\delta\bigr)\ln\frac{1/\rho\pm\delta}{\e\sigma}
+\frac{1}{\rho}\ln(\e\rho)
\Bigr]
\\
&\leq
\max_{\pm}\Bigl[
\bigl(\frac{1}{\rho}\pm\delta\bigr)
\ln\Bigl(\frac{1}{\sigma}\bigl(\frac{1}{\rho}\pm\delta\bigr)\Bigr)
-\bigl(\frac{1}{\rho}\pm\delta\bigr)
+\frac{1}{\rho}(1+\ln\rho)
\Bigr]
\\
&\leq
\frac{1}{\rho}
\max_{\pm}\Bigl[
\bigl(1\pm\delta\rho\bigr)
\ln\Bigl(\frac{1}{\sigma}\bigl(\frac{1}{\rho}\pm\delta\bigr)\Bigr)
\mp\delta\rho
+\ln\rho
\Bigr],
\end{align*}
for any $u,t\geq T_{\delta}$.

We can rewrite the estimate above as
\begin{equation*}
y_{\sigma}(u,t)+\frac{1}{\rho}\ln(\e\rho)
\leq\frac{1}{\rho}\max_{\pm}\tilde{y}(\pm\delta),
\end{equation*}
by defining
\begin{equation*}
\tilde{y}(\delta)
:=
\bigl(1+\delta\rho\bigr)
\ln\Bigl(\frac{1}{\sigma}\bigl(\frac{1}{\rho}+\delta\bigr)\Bigr)
-\delta\rho
+\ln\rho.
\end{equation*}
Note that for arbitrarily fixed $\rho$ and $\sigma$, the function
$\tilde{y}(\delta)$ is continuous in $\delta$ with $|\delta|<1/\rho$,
and satisfies $\tilde{y}(0)=-\ln\sigma$.
Therefore, for any given $\sigma'$ with $0<\sigma'<\sigma$, we have
$-\ln\sigma'>-\ln\sigma$, and we can choose a sufficiently small
$\delta:=\delta(\sigma')>0$ satisfying
\begin{equation*}
\max_{\pm}\tilde{y}(\pm\delta)\leq -\ln\sigma'.
\end{equation*}

Therefore, for any $\sigma'$ with $0<\sigma'<\sigma$,
we take such a choice of $\delta:=\delta(\sigma')$
and define $T_1:=T_{\delta(\sigma')}$.
Then, we have,
\begin{equation*}
y_{\sigma}(u,t)+\frac{1}{\rho}\ln(\e\rho)
\leq -\frac{1}{\rho}\ln\sigma',
\quad\text{for any $u,t\geq T_1$},
\end{equation*}
which concludes the proof.
\end{proof}

\section{Spaces of entire functions with growth given by a proximate order}
\label{sec:spac-entire-funct}

We use the standard notations around multi-indices.
For multi-indices 
$\alpha = (\alpha_1, \alpha_2, \cdots, \alpha_n)$,
$\beta = (\beta_1, \beta_2, \cdots, \beta_n) \in \NN^n$
with $\NN := \{ 0, 1, 2, 3, \cdots \}$
and $z = (z_1, z_2, \cdots, z_n) \in \CC^n$,
we set: 
\begin{gather*}
|\alpha| := \alpha_1 + \alpha_2 + \cdots + \alpha_n, \ 
\alpha ! := \alpha_1 ! \alpha_2 ! \cdots  \alpha_n !, \ 
\binom{\alpha}{\beta}
:= 
\frac{\alpha !}{(\alpha - \beta)! \beta !}, \\
|z| := \sqrt{|z_1|^2 + |z_2|^2 + \cdots + |z_n|^n}, \ 
\partial_z^{\alpha} 
:= \frac{\partial^{|\alpha|}}{\partial z_1^{\alpha_1} 
\partial z_2^{\alpha_2} \cdots \partial z_n^{\alpha_n}}.
\end{gather*}

Let $\varrho(r)$ be a proximate order for a positive order $\rho>0$.
For any $\sigma > 0$, we consider the Banach space
\begin{equation}\label{Arhosigma}
A_{\varrho,\sigma}
:=\{ f\in\calO(\CC^n) \mid
\norm{f}_{\varrho,\sigma}
:=\sup_{z\in\CC^n}\abs{f(z)}\exp(-\sigma\abs{z}^{\varrho(\abs{z})}) < \infty \}
\end{equation}
with the norm $\norm{\cdot}_{\varrho,\sigma}$.

In the sequel, for the simplicity, we will sometimes write
$\displaystyle{w_{\sigma} (z) := \sigma \abs{z}^{\varrho(|z|)}}$.
By the meaning of \cite[the proof of Lemme 1]{Mr},
we have:
\begin{lem}\label{compacticite}
If $\sigma_2 > \sigma_1 > 0$, the inclusion map 
$A_{\varrho,\sigma_1} \hookrightarrow A_{\varrho,\sigma_2}$ 
is compact.
\end{lem}
\begin{proof}
Set the unit ball with respect to $w_{\sigma_1}$:
$B := \{ f \in A_{\varrho,\sigma_1} \mid \norm{f}_{\varrho,\sigma_1} \leq 1 \}$ 
and we will show $B$ is relatively compact in $A_{\varrho,\sigma_2}$: 
for this, it is sufficient to prove
that any sequence $\{ f_k \} \subset B$  has an accumulation point.

First remark that $B$ is compact in $\calC(\CC^n)$ 
by the Ascoli-Arzel\`a Theorem:
in fact, for any compact convex set $K \subset \CC^n \simeq \RR^{2n}$ and $f \in B$, 
for $\forall x,y\in K$, we have
\begin{equation*}
|f(x)-f(y)| \leq C \cdot |x-y|
\end{equation*}
with $C := \sup_{u\in K} |\nabla f(u)|$. 
As $f \in B \subset A_{\varrho,\sigma_2} \subset \calO(\CC^n)$,
by the Cauchy's inequality, we have $\exists C_K > 0$ 
such that $\sup_{x\in K} |\nabla f(x)| \leq C_K$ $(f\in B)$
and thus $B$ is equicontinuous in $\calC(\CC^n)$.
As $B$ is also bounded in $\calC(\CC^n)$, 
we have the conclusion.

Therefore, (replacing with a sub-sequence),
we may assume $\{ f_k \} \subset B$ converges to $f \in B$ in $\calC(\CC^n)$
 with the compact convergence topology. 
\par
Next we will show $\forall \delta > 0$, $\exists R > 0$ such that
\begin{equation*}
\e^{w_{\sigma_1}(z)-w_{\sigma_2}(z)} < \frac{\delta}{2},
\quad\text{for $|z| > R$}.
\end{equation*}
This follows from
\begin{equation*}
w_{\sigma_2}(z) - w_{\sigma_1}(z)
= (\sigma_2 - \sigma_1)\abs{z}^{\varrho(|z|)} \rightarrow \infty
\quad\text{as $|z| \rightarrow \infty$}.
\end{equation*}
Thus we have for $|z| > R$,
\begin{align*}
|f_k(z)-f_l(z)| \e^{-w_{\sigma_2}(z)}
&= |f_k(z)-f_l(z)| \e^{-w_{\sigma_1}(z)} \cdot \e^{w_{\sigma_1}(z)-w_{\sigma_2}(z)}\\
&\leq \norm{f_k - f_l}_{\varrho,\sigma_1} \frac{\delta}{2}
\\
&\leq 2\cdot \frac{\delta}{2} = \delta.
\end{align*}
We know, from the convergence in $\calC(\CC^n)$
that
\begin{equation*}
\exists N > 0,\,
\forall k\geq \forall l\geq N,\,
\sup_{|z|\leq R} |f_k(z)-f_l(z)| \e^{-w_{\sigma_2}(z)}
\leq
\sup_{|z|\leq R} |f_k(z)-f_l(z)|
 < \delta.
\end{equation*}
Therefore for all $z \in \CC^n$,
we have $|f_k(z)-f_l(z)| \e^{-w_{\sigma_2}(z)} < \delta$,
or
\begin{equation*}
\norm{f_k - f_l}_{\varrho,\sigma_2} <\delta,
\end{equation*}
for $k \geq l \geq N$.
As $A_{\varrho,\sigma_2}$ is a Banach space, 
we have $f_k \to f$ in $A_{\varrho,\sigma_2}$.
\end{proof}

\begin{defn}[The spaces $A_{\varrho}$ and $A_{\varrho,+0}$]
\label{def:spaces-Arho-Arho0}
We define the space
\begin{equation}\label{calE}
A_{\varrho}
:=
\varinjlim_{\sigma>0}
A_{\varrho,\sigma}
\end{equation}
of entire functions at most {\em of normal type}
with respect to the proximate order $\varrho(r)$.
By the preceding Lemma~\ref{compacticite},
this space is a (DFS)-space,
which can be seen by taking an increasing sequence $(\sigma_j)$
tending to $+\infty$ instead of all $\sigma>0$.
See \cite[Definition 2.2.1 and Section 2.6]{B-G} for the one-variable case.
We also define, as in \cite{Mr},
\begin{equation*}
A_{\varrho,\sigma+0}
:=
\varprojlim_{\eps>0}
A_{\varrho,\sigma+\eps}
\end{equation*}
the locally convex space of entire functions at most of type $\sigma \geq 0$
with respect to a proximate order $\varrho(r)$,
which is (by taking a decreasing sequence $(\eps_j)$ tending to $0$ instead of all $\eps > 0$),
an (FS)-space by the preceding Lemma~\ref{compacticite}.
In particular, if $\sigma=0$, the space $A_{\varrho,+0}$ is
the locally convex space of entire functions at most {\em of minimal type}
with respect to the proximate order $\varrho(r)$.
\end{defn}

Let $\hat{\varrho}(r)$ be a normalization of $\varrho(r)$.
Throughout this section, we fix it and
define $\varphi(t)$ and $G_{\hat{\varrho},q}$ according to $\hat{\varrho}$.

Then, for any $\sigma>0$, the Banach space $A_{\hat{\varrho},\sigma}$
coincides with $A_{\varrho,\sigma}$
as subspaces of $\calO(\CC^n)$,
and $\norm{\cdot}_{\hat{\varrho},\sigma}$
becomes an equivalent norm with $\norm{\cdot}_{\varrho,\sigma}$.
In fact, it follows from \eqref{eq:rhohat-1},
that
\begin{equation}
\label{eq:diff-rhohat-rho}
c(\varrho,\hat{\varrho}):=\sup_{r>0}\abs{r^{\hat{\varrho}(r)}-r^{\varrho(r)}}
<+\infty,
\end{equation}
which implies
\begin{equation*}
\e^{-\sigma c(\varrho,\hat{\varrho})}\norm{f}_{\hat{\varrho},\sigma}
\leq\norm{f}_{\varrho,\sigma}
\leq
\e^{\sigma c(\varrho,\hat{\varrho})}\norm{f}_{\hat{\varrho},\sigma}.
\end{equation*}
Therefore, $A_{\hat{\varrho},\sigma}$ and $A_{\varrho,\sigma}$ are homeomorphic.

As a conclusion, the spaces $A_{\varrho}$ and $A_{\hat{\varrho}}$ coincide
and they share the same locally convex topologies as well,
and the same holds for $A_{\varrho,\sigma+0}$ and $A_{\hat{\varrho},\sigma+0}$.
In particular, an entire function is
of type $\sigma$ with respect to $\varrho(r)$
if and only if it is of the same type
with respect to a normalization of $\varrho(r)$.

We recall a theorem of \cite{Ll-G}:
\begin{thm}[{\cite[Theorem 1.23]{Ll-G}}]\label{type de f}
Let $f(z)=\sum f_{\alpha}z^{\alpha}\in\calO(\CC^n)$
be an entire function of finite order $\rho>0$ and of proximate order
$\varrho(r)$.
Then its type $\sigma$ with respect to $\varrho(r)$ is given by
\begin{equation}\label{type sigma}
\frac{1}{\rho} \ln \sigma
= \limsup_{q \to \infty}
\left(\frac{1}{q} \ln K_q + \ln \varphi(q) \right)
 - \frac{1}{\rho} - \frac{\ln \rho}{\rho}
\end{equation}
where $K_q := \sup_{|z|\leq 1} \abs{\sum_{\abs{\alpha}=q}f_{\alpha} z^{\alpha}}$.
\end{thm}

Using the above Theorem, we have several characterizations
for an entire function to belong to the spaces
given in Definition~\ref{def:spaces-Arho-Arho0}.
For the proof of Lemma~\ref{EstDeCoeffDef},
Corollary~\ref{EstimationDeChaqueCoeff}
and Proposition~\ref{prop:taylor-charact}
below,
we refer to \cite[Section 3]{I-J}.
\begin{lem}[{\cite[Lemma 2]{I-J}}]\label{EstDeCoeffDef}
An entire function $f(z) = \sum f_{\alpha} z^{\alpha} \in \calO(\CC^n)$
belongs to $A_{\varrho,\sigma+0}$ if and only if we have
\begin{equation}\label{CaracterisationParPrtieHomogene}
\limsup_{q \to \infty} \left( K_{q} G_{\hat{\varrho}, q} \right)^{\rho/q} \leq \sigma
\end{equation}
where $K_q := \sup_{|z|\leq 1} \abs{\sum_{\abs{\alpha}=q}f_{\alpha} z^{\alpha}}$. 
\end{lem}

\begin{cor}[{\cite[Corollary 1]{I-J}}]\label{EstimationDeChaqueCoeff}
If $f(z) = \sum f_{\alpha} z^{\alpha} \in \calO(\CC^n)$
belongs to $f\in A_{\varrho,\sigma+0}$, then we have
\begin{equation}\label{CaracterisationParCoefficients}
\limsup_{q \to \infty}
\Bigl(\max_{|\alpha|=q}|f_{\alpha}| G_{\hat{\varrho}, q}\Bigr)^{\rho/q}
\leq \sqrt{n}^{\rho} \sigma.
\end{equation}
Conversely, if $\{ f_{\alpha} \}$ satisfies \eqref{CaracterisationParCoefficients},
then we have  $f(z) \in A_{\varrho,\sqrt{n}^{\rho}\sigma+0}$.
\end{cor}

By the corollary, we have
\begin{prop}[{\cite[Proposition 1]{I-J}}]\label{prop:taylor-charact}
An entire function $f(z) = \sum f_{\alpha} z^{\alpha} \in \calO(\CC^n)$
belongs to $A_{\varrho}$ if and only if we have
\begin{equation}\label{CarParCoefficientsPourCalE}
\limsup_{q \to \infty}
\Bigl(\max_{|\alpha|=q}|f_{\alpha}| G_{\hat{\varrho},q}\Bigr)^{\rho/q} < \infty.
\end{equation}
\end{prop}

We prepare some estimates of norms of monomials.
(C.f. \cite[Lemma 5.1]{I-M}).
\begin{lem}\label{normDUnMonome2}
Suppose $\rho>0$.
For any $\sigma$ and $\sigma'$ with $0<\sigma'<\sigma$,
there exists $C > 0$ such that for any $\beta\in\NN^n$, we have 
\begin{equation}\label{eq:norm-monomial}
\norm{z^{\beta}}_{\varrho,\sigma}
\leq C {\sigma'}^{-\abs{\beta}/\rho} G_{\hat{\varrho},|\beta|}.
\end{equation}
\end{lem}
\begin{proof}
In this proof, we use the abbreviation $G_q$ for $G_{\hat{\varrho},q}$.

Since the norms $\norm{\cdot}_{\varrho,\sigma}$ and
$\norm{\cdot}_{\hat{\varrho},\sigma}$ are equivalent,
we may assume from the beginning that
$\varrho(r)$ satisfies \eqref{eq:rhohat-2}
and $\hat{\varrho}(r)\equiv\varrho(r)$.

We will first prove
\begin{equation*}
|z^{\beta} \e^{- w_{\sigma} (z)}|
\leq {\sigma'}^{-\abs{\beta}/\rho} G_{|\beta|}
\end{equation*}
for sufficiently large $\abs{z}$ and $|\beta|$.
Recall that $r=\varphi(t)$ denotes the inverse function of $t=r^{\varrho(r)}$.
Setting $q := |\beta|$ and $r=\abs{z}=\varphi(t)$, we have
\begin{align*}
|z^{\beta} \e^{- w_{\sigma}(z)}| G_q^{-1}
&\leq
r^q \e^{- \sigma r^{\varrho(r)}} G_q^{-1} \\
&= \exp\bigl(q \ln \varphi(t) - \sigma t - \ln G_q\bigr) \\
&= \exp\Bigl[q \bigl\{ \ln \varphi(t) - \sigma \frac{t}{q}
- \ln \varphi(q) + \frac{1}{\rho} \ln (\e \rho) \bigr\}\Bigr]\\
&= \exp\Bigl[q \bigl\{ \ln \frac{\varphi(t)}{\varphi(q)} - \sigma \frac{t}{q}
 + \frac{1}{\rho} \ln (\e \rho) \bigr\}\Bigr].
\end{align*}
Applying Lemma~\ref{lem:y-u-t} to
\begin{equation*}
y_{\sigma}(u,t):=\ln\frac{\varphi(t)}{\varphi(u)}-\sigma\frac{t}{u}
\end{equation*}
and $\sigma'$ with $0<\sigma'<\sigma$,
we get $T_1>0$ such that
\begin{equation*}
y_{\sigma}(u,t)+\frac{1}{\rho}\ln(\e\rho)
\leq -\frac{1}{\rho}\ln\sigma',
\quad\text{for any $u,t\geq T_1$},
\end{equation*}
and therefore that
\begin{equation}
\label{eq:exterior}
|z^{\beta} \e^{- w_{\sigma}(z)}| G_q^{-1}
\leq
\exp\Bigl(-\frac{q}{\rho}\ln\sigma'\Bigr)
=\sigma'^{-q/\rho},
\quad
\text{for $|\beta|=q\geq T_1$, $\abs{z}\geq \varphi(T_1)$}.
\end{equation}

Next we consider the case $|z|\leq \varphi(T_1)$, $|\beta|=q \gg 1$.
Similarly, we have, for $|z|=r$,
\begin{align*}
|z^{\beta} \e^{- w_{\sigma}(z)}| G_q^{-1}
&\leq
r^q \e^{- \sigma r^{\varrho(r)}} G_q^{-1} \\
&= \exp\bigl(q \ln r - \sigma r^{\varrho(r)} - \ln G_q\bigr) \\
&= \exp\Bigl[q \bigl\{ \ln r - \sigma \frac{r^{\varrho(r)}}{q} 
- \ln \varphi(q) + \frac{1}{\rho} \ln (\e \rho) \bigr\}\Bigr]\\
&= \exp\Bigl[q \bigl\{ \ln \frac{\varphi(t)}{\varphi(q)}
 - \sigma \frac{t}{q} + \frac{1}{\rho} \ln (\e \rho) \bigr\}\Bigr]
\\
&\leq \exp\Bigl[
q \bigl\{ \ln \frac{\varphi(T_1)}{\varphi(q)}
+ \frac{1}{\rho}\ln (\e \rho) \bigr\}
\Bigr].
\end{align*}
Here at the last inequality,
we used $r=\varphi(t)\leq\varphi(T_1)$
and neglected the non-positive term $-\sigma t/q$.
Note that for any given $\sigma'>0$, we can take $T_2\geq T_1$
such that
\begin{equation*}
\ln \frac{\varphi(T_1)}{\varphi(T_2)}
+ \frac{1}{\rho}\ln (\e \rho)
\leq -\frac{1}{\rho}\ln\sigma',
\end{equation*}
and therefore that
\begin{equation}
\label{eq:interior}
|z^{\beta} \e^{- w_{\sigma}(z)}| G_q^{-1}
\leq\sigma'^{-q/\rho},
\quad
\text{for $|\beta|=q\geq T_2$, $\abs{z}\leq \varphi(T_1)$}.
\end{equation}

These two inequalities \eqref{eq:exterior} and \eqref{eq:interior} imply
\begin{equation*}
|z^{\beta} \e^{- w_{\sigma}(z)}| G_q^{-1}
\leq\sigma'^{-q/\rho},
\quad
\text{for $|\beta|=q\geq T_2$, $z\in\CC^n$},
\end{equation*}
or,
\begin{equation*}
\norm{z^{\beta}}_{\varrho,\sigma} 
\leq {\sigma'}^{-|\beta|/\rho} G_{|\beta|},
\quad
\text{for $|\beta|\geq T_2$}.
\end{equation*}
We know already that $ \norm{z^{\beta}}_{\varrho,\sigma}<+\infty$ for any $\beta$.
Therefore, under the choice of $C$ with
\begin{equation*}
C\geq
\max\Bigl\{1,
\max_{|\beta|<T_2}
\norm{z^{\beta}}_{\varrho,\sigma} G_{|\beta|}^{-1} {\sigma'}^{|\beta|/\rho}
\Bigr\},
\end{equation*}
we get the estimate \eqref{eq:norm-monomial}.
\end{proof}

We also prepare some estimates of differential operators.
\begin{lem}\label{lem:norm-differentiation}
There exists a constant $\kappa$ depending only on $\rho$ for which the
following statement holds:
For any $\sigma>0$, we can take $C(\sigma)$ such that
for any $f\in A_{\hat{\varrho},\sigma}$, and any $\alpha\in\NN^n$,
the inequality
\begin{equation*}
\frac{1}{\alpha!}\norm{\partial_z^{\alpha}f(z)}_{\hat{\varrho},\kappa\sigma}
\leq C(\sigma)\norm{f}_{\hat{\varrho},\sigma}
\frac{(2 \kappa n^{\rho/2}\sigma)^{q/\rho}}{G_{\hat{\varrho},q}}
\end{equation*}
holds.
Here we denote $q=\abs{\alpha}$.
\end{lem}
(C.f. \cite[Lemma 2.1]{aiosu}).
\begin{proof}
First we apply Lemma~\ref{lem:pseudo-subadditivity} to $\hat{\varrho}(r)$,
and get constants $\kappa$ and $B$ such that
\begin{equation*}
(r+s)^{\hat{\varrho}(r+s)}
\leq \kappa(r^{\hat{\varrho}(r)}+s^{\hat{\varrho}(s)})+B,
\quad\text{for $r,s>0$}.
\end{equation*}
Note that $\kappa$ depends only on $\rho$ while $B$ depends on $\hat{\varrho}(r)$.
The Cauchy estimate gives us, for $\abs{z}\leq r$,
$\abs{\alpha}=q$,
\begin{align*}
\frac{\abs{\partial_z^{\alpha}f(z)}}{\alpha!}
&\leq\inf_{s>0}\frac{1}{(s/\sqrt{n})^{\abs{\alpha}}}
\max_{\abs{\zeta_i}=s/\sqrt{n}}\abs{f(z+\zeta)}
\\
&\leq
\norm{f}_{\hat{\varrho},\sigma}
\inf_{s>0}\frac{n^{q/2}}{s^{q}}\exp(\sigma(r+s)^{\hat{\varrho}(r+s)})
\\
&\leq
\norm{f}_{\hat{\varrho},\sigma}
\inf_{s>0}\frac{n^{q/2}}{s^{q}}
\exp\bigl(\kappa\sigma(r^{\hat{\varrho}(r)}+s^{\hat{\varrho}(s)})+B\sigma\bigr)
\\
&=
\e^{B\sigma}\norm{f}_{\hat{\varrho},\sigma}\exp(\kappa\sigma r^{\hat{\varrho}(r)})
n^{q/2}
\inf_{s>0}\frac{\exp(\kappa\sigma s^{\hat{\varrho}(s)})}{s^{q}}.
\end{align*}
The infimum can be estimated as
\begin{align*}
\inf_{s>0}\frac{\exp(\kappa\sigma s^{\hat{\varrho}(s)})}{s^{q}}
&\leq
\inf_{t>0}\frac{\exp(\kappa\sigma t)}{\varphi(t)^{q}}
\leq
\frac{\exp(\kappa\sigma t)}{\varphi(t)^{q}}\Bigr|_{t=q/(\kappa\sigma\rho)}
\\
&\leq
\frac{\e^{q/\rho}}{\varphi(q/(\kappa\sigma\rho))^q}.
\end{align*}
Therefore, we have
\begin{align*}
\frac{\abs{\partial_z^{\alpha}f(z)}}{\alpha!}
&\leq
\e^{B\sigma}\norm{f}_{\hat{\varrho},\sigma}\exp(\kappa\sigma r^{\hat{\varrho}(r)})
\frac{(n^{\rho/2}\e)^{q/\rho}}{\varphi(q)^q}\cdot
\Bigl(\frac{\varphi(q)}{\varphi(q/(\kappa\sigma\rho))}\Bigr)^q.
\end{align*}
It follows from
\begin{equation*}
\lim_{q\rightarrow\infty}\frac{\varphi(q)}{\varphi(q/(\kappa\sigma\rho))}
=(\kappa\sigma\rho)^{1/\rho}
\end{equation*}
(see \cite[p.42, (1.58)]{Lv}),
that there exists $C(\sigma)$ depending on $\sigma$
and $\hat{\varrho}$ such that
\begin{equation*}
\forall q,\quad
\e^{B\sigma}\Bigl(\frac{\varphi(q)}{\varphi(q/(\kappa\sigma\rho))}\Bigr)^q
\leq C(\sigma)(2 \kappa\sigma\rho)^{q/\rho}.
\end{equation*}
Therefore, we have
\begin{align*}
\frac{\abs{\partial_z^{\alpha}f(z)}}{\alpha!}
&\leq
C(\sigma)\norm{f}_{\hat{\varrho},\sigma}\exp(\kappa\sigma r^{\hat{\varrho}(r)})
\frac{(n^{\rho/2}\e)^{q/\rho}}{\varphi(q)^q}
(2 \kappa\sigma\rho)^{q/\rho}
\\
&\leq
C(\sigma)\norm{f}_{\hat{\varrho},\sigma}\exp(\kappa\sigma r^{\hat{\varrho}(r)})
\frac{(2 \kappa n^{\rho/2}\sigma)^{q/\rho}}{G_{\hat{\varrho},q}},
\end{align*}
which conclude the proof.
\end{proof}

By these preparations, we give
\begin{prop}\label{prop:taylor-exp-converge}
For an entire function $f(z)$ belonging to $A_{\varrho,\sigma+0}$,
its Taylor expansion $\sum_{\alpha\in\NN^n}f_{\alpha}z^{\alpha}$
converges to $f(z)$ in the space $A_{\varrho,\sqrt{n}^{\rho}\sigma+0}$.

In particular, the set of polynomials $\CC[z]$
is dense in $A_{\varrho,+0}$ and also dense in $A_{\varrho}$.
\end{prop}
\begin{proof}
For the former statement, it suffices to show that
$\sum_{\alpha\in\NN^n}\norm{f_{\alpha}z^{\alpha}}_{\varrho,\sqrt{n}^{\rho}(\sigma+\eps)}$
is finite for any $\eps>0$.

By Lemma~\ref{normDUnMonome2}, there exists a constant $C_0$ such that
\begin{equation*}
\norm{z^{\alpha}}_{\varrho,\sqrt{n}^{\rho}(\sigma+\eps)}
\leq C_0\bigl(\sqrt{n}^{\rho}(\sigma+\eps/2)\bigr)^{-\abs{\alpha}/\rho}
G_{\hat{\varrho},\abs{\alpha}},
\quad
\text{for any $\alpha\in\NN^n$}.
\end{equation*}
On the other hand, by Corollary~\ref{EstimationDeChaqueCoeff},
there is a constant $C_1$ such that
\begin{equation*}
\max_{\abs{\alpha}=q}\abs{f_{\alpha}}G_{\hat{\varrho},q}
\leq C_1\bigl(\sqrt{n}^{\rho}(\sigma+\eps/4)\bigr)^{q/\rho},
\quad
\text{for any $\alpha\in\NN^n$}.
\end{equation*}
Therefore, we have
\begin{align*}
\sum_{\alpha\in\NN^n}
\norm{f_{\alpha}z^{\alpha}}_{\varrho,\sqrt{n}^{\rho}(\sigma+\eps)}
&=\sum_{q\in\NN}\sum_{\abs{\alpha}=q}
\abs{f_{\alpha}}G_{\hat{\varrho},q}
\cdot\norm{z^{\alpha}}_{\varrho,\sqrt{n}^{\rho}(\sigma+\eps)}/G_{\hat{\varrho},q}
\\
&\leq
\sum_{q\in\NN}\sum_{\abs{\alpha}=q}
C_0\bigl(\sqrt{n}^{\rho}(\sigma+\eps/2)\bigr)^{-q/\rho}
C_1\bigl(\sqrt{n}^{\rho}(\sigma+\eps/4)\bigr)^{q/\rho}
\\
&\leq
C_0 C_1
\sum_{q\in\NN}
{}_{n}H_q
\Bigl(\frac{\sigma+\eps/4}{\sigma+\eps/2}\Bigr)^{q/\rho}
\\
&<+\infty.
\end{align*}
Here we used $0<\frac{\sigma+\eps/4}{\sigma+\eps/2}<1$
and
\begin{equation}
\label{eq:multi-choose}
{}_{n}H_q:=\sum_{\abs{\alpha}=q}1=\binom{n+q-1}{q}\leq(q+1)^{n-1}.
\end{equation}

For the latter statement in the case $f\in A_{\varrho,+0}$,
it follows from the former with $\sigma=0$ that
\begin{equation*}
\lim_{q\rightarrow\infty}\sum_{\abs{\alpha}\leq q}f_{\alpha}z^{\alpha}=f(z),
\end{equation*}
in the space $A_{\varrho,+0}$.

In the case $f\in A_{\varrho}$, there exists $\sigma$
such that $f\in A_{\varrho,\sigma+0}$.
Then the same convergence holds
in the space $f\in A_{\varrho,\sqrt{n}^{\rho}\sigma+0}$,
and therefore in the space $f\in A_{\varrho}$.
\end{proof}

\section{differential operator representations of continuous homomorphisms}
\label{sec:diff-oper-repr}

Before studying continuous homomorphisms between $A_{\varrho_i}$
($i=1,2$) and those between $A_{\varrho_i,+0}$ ($i=1,2$),
we prepare a differential operator representation
of homomorphisms from $\CC[z]$ to $\CC[[z]]$.
We define the space of formal differential operators of infinite order
with coefficients in $\CC[[z]]$ by
\begin{equation*}
\hat{D}:=\{P=\sum_{\alpha\in\NN^n}a_{\alpha}(z)\partial_z^{\alpha}\mid
a_{\alpha}(z)\in\CC[[z]]\}.
\end{equation*}
Note that $\hat{D}$ is linearly isomorphic to
$\prod_{\alpha\in\NN^n}\CC[[z]]$,
by the correspondence
\begin{equation*}
\sum_{\alpha\in\NN^n}a_{\alpha}(z)\partial_z^{\alpha}
\mapsto(a_{\alpha}(z))_{\alpha\in\NN^n}.
\end{equation*}

\begin{prop}\label{prop:linear-poly-fps}
There are two linear isomorphisms:
\begin{equation*}
\hat{D}
\xrightarrow{\sim}
\mathop{\mathrm{Hom}}\nolimits_{\CC}(\CC[z],\CC[[z]])
\xrightarrow{\sim}
\prod_{\beta\in\NN^n}\CC[[z]],
\end{equation*}
where the first and second mappings are given by
\begin{gather*}
\sum_{\alpha\in\NN^n}a_{\alpha}(z)\partial_z^{\alpha}\mapsto
\Bigl(\CC[z]\ni\sum_{\gamma}f_{\gamma}z^\gamma\mapsto
\sum_{\alpha\leq\gamma}a_{\alpha}(z)
\frac{f_{\gamma}\gamma!z^{\gamma-\alpha}}{(\gamma-\alpha)!}
\in\CC[[z]]\Bigr),
\\
F\mapsto(F z^{\beta}/\beta!)_{\beta\in\NN^n},
\end{gather*}
respectively.
\end{prop}
\begin{proof}
We can easily see that both mappings are linear mappings between vector spaces
and that their composition is given by
\begin{equation*}
\sum_{\alpha}a_{\alpha}(z)\partial_z^{\alpha}
\mapsto
\Bigl(\sum_{\alpha\leq\beta}a_{\alpha}(z)
\frac{z^{\beta-\alpha}}{(\beta-\alpha)!}\Bigr)_{\beta\in\NN^n}.
\end{equation*}
Therefore, it suffices to show that the relation
\begin{equation}\label{eq:b-beta-by-a-alpha}
b_{\beta}(z)=\sum_{\alpha\leq\beta}a_{\alpha}(z)
\frac{z^{\beta-\alpha}}{(\beta-\alpha)!},
\quad
\beta\in\NN^n,
\end{equation}
induces a bijection
\begin{equation*}
\prod_{\alpha\in\NN^n}\CC[[z]]\ni
(a_{\alpha}(z))_{\alpha\in\NN^n}
\mapsto
(b_{\beta}(z))_{\beta\in\NN^n}
\in\prod_{\beta\in\NN^n}\CC[[z]].
\end{equation*}
This follows from the fact that the relation \eqref{eq:b-beta-by-a-alpha}
can be inverted as
\begin{equation}\label{eq:a-alpha-by-b-beta}
a_{\alpha}(z)=\sum_{\beta\leq\alpha}b_{\beta}(z)
\frac{(-z)^{\alpha-\beta}}{(\alpha-\beta)!},
\quad
\alpha\in\NN^n.
\end{equation}
(C.f. \cite[the proof of Theorem 3.3]{aiosu}).
In fact, we can calculate
the $\gamma$-element of
the image of $(a_{\alpha}(z))_{\alpha}$
by the composition of \eqref{eq:a-alpha-by-b-beta}
and \eqref{eq:b-beta-by-a-alpha} as
\begin{align*}
\sum_{\beta\leq\gamma}
\sum_{\alpha\leq\beta}a_{\alpha}(z)\frac{z^{\beta-\alpha}}{(\beta-\alpha)!}
\frac{(-z)^{\gamma-\beta}}{(\gamma-\beta)!}
=\sum_{\alpha\leq\gamma}a_{\alpha}(z)
\frac{(z-z)^{\gamma-\alpha}}{(\gamma-\alpha)!}
=a_{\gamma}(z),
\end{align*}
which implies the composition is the identity.
We can similarly show that the composition of \eqref{eq:b-beta-by-a-alpha}
and \eqref{eq:a-alpha-by-b-beta} is the identity.
\end{proof}

Now we study
continuous homomorphisms from $A_{\varrho_1}$ to $A_{\varrho_2}$
and those from $A_{\varrho_1,+0}$ to $A_{\varrho_2,+0}$,
where $\varrho_i$ ($i=1,2$) are two proximate orders for positive orders
$\rho_i=\lim_{r\rightarrow\infty}\varrho_i(r)>0$
satisfying
\begin{equation}
\label{eq:varrho1-is-O-varrho2}
r^{\varrho_1(r)}=O(r^{\varrho_2(r)}),
\quad\text{as $r\rightarrow\infty$}.
\end{equation}
For such $\varrho_i$ ($i=1,2$), we define two subspaces of $\hat{D}$.
\begin{defn}\label{def:D-rho1-rho2}
Let $\varrho_i$ ($i=1,2$) be two proximate orders for orders
$\rho_i>0$ satisfying \eqref{eq:varrho1-is-O-varrho2}.
We take a normalization $\hat{\varrho}_1$ of $\varrho_1$
as in Definition~\ref{def:normaliz-prox} and $G_{\hat{\varrho}_1,q}$
by \eqref{G-rho-q}.
We denote by $\boldsymbol{D}_{\varrho_1\rightarrow\varrho_2}$
and by $\boldsymbol{D}_{\varrho_1\rightarrow\varrho_2,0}$
the sets of all formal differential operator $P$ of the form
\begin{equation}\label{inf-ord}
P=\sum_{\alpha\in\NN^n}a_{\alpha}(z)\partial_z^{\alpha},
\end{equation}
where the multisequence
$(a_{\alpha}(z))_{\alpha\in\NN^n}\subset A_{\varrho_2}$
satisfies
\begin{equation}
\label{eq:est-D-rho1-rho2-normal}
\forall\lambda>0,\,
\exists\sigma>0,\,
\exists C>0,\,
\forall\alpha\in\NN^n,\,
\norm{a_{\alpha}}_{\varrho_2,\sigma}
\leq C\frac{G_{\hat{\varrho}_1,\abs{\alpha}}}{\alpha!}\lambda^{\abs{\alpha}},
\end{equation}
and
\begin{equation}
\label{eq:est-D-rho1-rho2-minimal}
\forall\sigma>0,\,
\exists \lambda>0,\,
\exists C>0,\,
\forall\alpha\in\NN^n,\,
\norm{a_{\alpha}}_{\varrho_2,\sigma}
\leq C\frac{G_{\hat{\varrho}_1,\abs{\alpha}}}{\alpha!}\lambda^{\abs{\alpha}},
\end{equation}
respectively.
Note that in the latter case,
each $a_{\alpha}$ belongs to $A_{\varrho_2,+0}$.
\end{defn}

For the case $\varrho_1=\varrho_2$, we denote
$\boldsymbol{D}_{\varrho\rightarrow\varrho}$ by
$\boldsymbol{D}_{\varrho}$,
and $\boldsymbol{D}_{\varrho\rightarrow\varrho,0}$ by
$\boldsymbol{D}_{\varrho,0}$.
That is,
\begin{defn}\label{def:D-rho}
Let $\varrho$ be a proximate order for an order $\rho>0$.
We take a normalization $\hat{\varrho}$ and a sequence $G_{\hat{\rho},q}$,
given by Definition~\ref{def:normaliz-prox} and \eqref{G-rho-q}.
We denote by $\boldsymbol{D}_{\varrho}$
and by $\boldsymbol{D}_{\varrho,0}$
the sets of all formal differential operator $P$ of the form
\begin{equation*}
P=\sum_{\alpha\in\NN^n}a_{\alpha}(z)\partial_z^{\alpha},
\end{equation*}
where the multisequence
$(a_{\alpha}(z))_{\alpha\in\NN^n}\subset A_{\varrho}$
satisfies
\begin{equation}
\label{eq:est-D-rho-normal}
\forall\lambda>0,\,
\exists\sigma>0,\,
\exists C>0,\,
\forall\alpha\in\NN^n,\,
\norm{a_{\alpha}}_{\varrho,\sigma}
\leq C\frac{G_{\hat{\varrho},\abs{\alpha}}}{\alpha!}\lambda^{\abs{\alpha}},
\end{equation}
and
\begin{equation}
\label{eq:est-D-rho-minimal}
\forall\sigma>0,\,
\exists \lambda>0,\,
\exists C>0,\,
\forall\alpha\in\NN^n,\,
\norm{a_{\alpha}}_{\varrho,\sigma}
\leq C\frac{G_{\hat{\varrho},\abs{\alpha}}}{\alpha!}\lambda^{\abs{\alpha}},
\end{equation}
respectively.
Note also that in the latter case,
each $a_{\alpha}$ belongs to $A_{\varrho,+0}$.
\end{defn}
Note that the definitions above are well-defined, that is,
they do not depend on the choice of normalizations of $\varrho_1$ and $\varrho$.

\begin{rem}\label{rem:replace-prox-orders}
By adding $\ln c/\ln r$ for a constant $c>0$
(with a suitable modification near $r=0$)
to a proximate order $\varrho_2(r)$ with order $\rho_2$,
we get a new proximate order $\tilde{\varrho}_2(r)$ for the same order $\rho_2$
satisfying $\tilde{\varrho}_2(r)=\varrho_2(r)+\ln c/\ln r$ for $r\gg 1$,
that is, $r^{\tilde{\varrho}_2(r)}=c r^{\varrho_2(r)}$ eventually.
Then $\norm{\cdot}_{\tilde{\varrho}_2,\sigma}$
and $\norm{\cdot}_{\varrho_2,c\sigma}$
become equivalent norms for $\sigma>0$,
and the spaces $A_{\tilde{\varrho}_2}$ and $A_{\tilde{\varrho}_2,+0}$
are homeomorphic to $A_{\varrho_2}$ and $A_{\varrho_2,+0}$
respectively.
By taking sufficiently large $c$, we can take $\tilde{\varrho}_2$ as
\begin{equation*}
\varrho_1(r)\leq \tilde{\varrho}_2(r),
\quad\text{for $r\geq r_0$}
\end{equation*}
for a suitable $r_0$,
and we can choose
normalizations $\hat{\varrho}_1$ of $\varrho_1$
and $\hat{\varrho}_2$ of $\tilde{\varrho}_2$ as
\begin{itemize}
\item[(i)]
$\hat{\varrho}_1(r)\leq\hat{\varrho}_2(r)$ for $r\geq 0$.
\end{itemize}
Since a proximate order and its normalization define equivalent norms,
we have
\begin{itemize}
\item[(ii)]
$\norm{\cdot}_{\hat{\varrho}_2,\sigma}$
and $\norm{\cdot}_{\varrho_2,c\sigma}$ are equivalent
for any $c>0$,
\item[(iii)]
$\boldsymbol{D}_{\varrho_1\rightarrow\varrho_2}
=\boldsymbol{D}_{\hat{\varrho}_1\rightarrow\hat{\varrho}_2}$,
$\boldsymbol{D}_{\varrho_1\rightarrow\varrho_2,0}
=\boldsymbol{D}_{\hat{\varrho}_1\rightarrow\hat{\varrho}_2,0}$.
\end{itemize}
Note further that
Theorems~\ref{thm:main1} and \ref{thm:main2} below are not affected
by the replacement of $\varrho_1$ and $\varrho_2$
by $\hat{\varrho}_1$ and $\hat{\varrho}_2$.
\end{rem}

Now we will prove:
\begin{thm}\label{thm:main1}
Let $\rho_i$ $(i=1,2)$ be two proximate orders for orders $\rho_i>0$
satisfying \eqref{eq:varrho1-is-O-varrho2}.

\noindent{\rm (i)}
Suppose that $P\in\boldsymbol{D}_{\varrho_1\rightarrow\varrho_2}$
has the form \eqref{inf-ord}.
For an entire function $f\in A_{\varrho_1}$,
\[
Pf:=\sum_\alpha a_\alpha(z)\partial_z^\alpha f
\]
converges and $Pf\in A_{\varrho_2}$.
Moreover, $f\mapsto Pf$ defines a continuous homomorphism
$P:A_{\varrho_1}\longrightarrow A_{\varrho_2}$.

\noindent{\rm (ii)}
Let $F:A_{\varrho_1}\longrightarrow A_{\varrho_2}$
be a continuous homomorphism.
Then there is a unique $P\in\boldsymbol{D}_{\varrho_1\rightarrow\varrho_2}$
such that $Ff=Pf$ holds for any $f\in A_{\varrho_1}$.
\end{thm}
\begin{proof}
We can replace $\varrho_i$ by $\hat{\varrho}_i$ ($i=1,2$)
as in Remark~\ref{rem:replace-prox-orders},
and we may assume from the beginning
that $\varrho_i$ ($i=1,2$) are normalized proximate orders satisfying
\begin{equation*}
\varrho_1(r)\leq\varrho_2(r),
\quad\text{for $r\geq 0$},
\end{equation*}
which implies
\begin{equation}
\label{eq:norm-varrho-2-leq-1}
\norm{f}_{\varrho_2,\tau}\leq\norm{f}_{\varrho_1,\tau}
\end{equation}
for any $f$ and $\tau>0$.

(i)
Using Lemma~\ref{lem:norm-differentiation}
and the estimate~\eqref{eq:est-D-rho1-rho2-normal} for $(a_\alpha(z))_{\alpha\in\NN^n}$
in Definition~\ref{def:D-rho1-rho2},
we have a constant $\kappa=\kappa(\rho_1)$ depending only on $\rho_1$
such that
for any $\varepsilon>0$, $\tau>0$
there exists positive constants
$\sigma(\varepsilon)$, $C(\varepsilon)$ and $C'(\tau)$
with the estimate
\begin{align}
\label{Pfnorm}
&\sum_{\alpha\in\NN^n}
\norm{a_{\alpha}\partial_z^{\alpha}f}_{\varrho_2,\sigma(\eps)+\kappa\tau}
\\
\nonumber
&\leq
\sum_{\alpha\in\NN^n}\norm{a_{\alpha}}_{\varrho_2,\sigma(\eps)}
\norm{\partial_z^\alpha f}_{\varrho_2,\kappa\tau}
\\
\nonumber
&
\leq
\sum_{\alpha\in\NN^n}\norm{a_{\alpha}}_{\varrho_2,\sigma(\eps)}
\norm{\partial_z^\alpha f}_{\varrho_1,\kappa\tau}
\\
\nonumber
&\leq
\sum_{\alpha\in\NN^n}
C(\eps)\frac{G_{\varrho_1,\abs{\alpha}}}{\alpha!}\varepsilon^{\abs{\alpha}}
\cdot
C'(\tau)\norm{f}_{\varrho_1,\tau}
\frac{\alpha!}{G_{\varrho_1,\abs{\alpha}}}
\bigl(2\kappa n^{\rho_1/2}\tau\bigr)^{\abs{\alpha}/\rho_1}
\\
\nonumber
&=
C(\eps)C'(\tau)
\norm{f}_{\varrho_1,\tau}
\sum_{q\in\NN}
\sum_{\abs{\alpha}=q}
\eps^q
\bigl(2\kappa n^{\rho_1/2}\tau\bigr)^{q/\rho_1}
\\
\nonumber
&=
C(\eps)C'(\tau)
\norm{f}_{\varrho_1,\tau}
\sum_{q\in\NN}
{}_{n}H_q
\bigl(\sqrt{n}(2\kappa\tau)^{1/\rho_1}\eps\bigr)^q
\end{align}
for any $f\in A_{\varrho_1,\tau}$.
Here we have used
\begin{equation}
\label{eq:norm-submultiplicative}
\norm{g_1 g_2}_{\varrho,\tau_1+\tau_2}
\leq
\norm{g_1}_{\varrho,\tau_1}
\norm{g_2}_{\varrho,\tau_2},
\quad\text{for $g_i\in A_{\varrho,\tau_i}$, $(i=1,2)$},
\end{equation}
at the first inequality
and \eqref{eq:norm-varrho-2-leq-1} at the second.
Note that $\sigma(\eps)$, $C(\eps)$ and $C'(\tau)$ are independent of $f$.
It follows from \eqref{eq:multi-choose}
that the last sum in \eqref{Pfnorm} converges
if $\eps<(\sqrt{n}(2\kappa\tau)^{1/\rho_1})^{-1}$.
For such a choice of $\varepsilon>0$ depending on $\rho_1$ and $\tau$,
we set $\tau'(\tau)=\sigma(\eps)+\kappa\tau$ and
\begin{equation*}
C''=C(\eps)C'(\tau)
\sum_{q\in\NN}
{}_{n}H_q
\bigl(\sqrt{n}(2\kappa\tau)^{1/\rho_1}\eps\bigr)^q.
\end{equation*}
Then, $\sum_{\alpha}a_{\alpha}\partial_z^{\alpha}f$ converges in
$A_{\varrho,\tau'(\tau)}$
and defines an element $Pf\in A_{\varrho,\tau'(\tau)}$ satisfying
\begin{equation*}
\norm{Pf}_{\varrho_2,\tau'(\tau)}\leq C''\norm{f}_{\varrho,\tau}.
\end{equation*}
Since $f\in A_{\varrho,\tau}$ was chosen arbitrarily,
this implies the well-definedness and the continuity of
$P:A_{\varrho_1,\tau}\longrightarrow A_{\varrho_2,\tau'(\tau)}$.
Also since $\tau>0$ was chosen arbitrarily,
we get the well-definedness and the continuity
of $P:A_{\varrho_1}\longrightarrow A_{\varrho_2}$
by the definition of the inductive limit of locally convex spaces.

(ii) Let $F:A_{\varrho_1}\longrightarrow A_{\varrho_2}$ be
a continuous homomorphism.
Then, thanks to the theory of locally convex spaces,
we can conclude, using Lemma~\ref{compacticite},
that for any $\tau>0$, there exist $\tau'=\tau'(\tau)>0$
such that $F(A_{\varrho_1,\tau})\subset A_{\varrho_2,\tau'(\tau)}$
and that
\begin{equation*}
F:A_{\varrho_1,\tau}\rightarrow A_{\varrho_2,\tau'(\tau)}
\end{equation*}
is continuous.
(Refer to \cite[Chap. 3, \S 1, Proposition 7]{B-EVT},
or \cite[Chap 4, Part 1, 5, Corollary 1]{G-TVS}).
Therefore, we can in particular take $C(\tau)$ depending on $\tau>0$
for which
\begin{equation}\label{conti-est1}
\norm{F f}_{\varrho_2,\tau'(\tau)}\leq C(\tau)\norm{f}_{\varrho_1,\tau},
\quad\text{for any $f\in A_{\varrho_1,\tau}$}.
\end{equation}

Let us define a multisequence of entire functions $(a_{\alpha}(z))_{\alpha\in\NN^n}$ by
\begin{equation}\label{a_alpha-def}
a_\alpha(z)=\sum_{\beta\le\alpha}\frac{(-z)^{\alpha-\beta}}{(\alpha-\beta)!}
\frac{F z^\beta}{\beta!},
\end{equation}
whose convergence in $A_{\varrho_2}$ will be proved
together with their estimates.
We define a formal differential operator $P\in\hat{D}$ of infinite order by
\begin{equation}\label{defP}
P=\sum_\alpha a_\alpha(z)\partial_z^\alpha.
\end{equation}
First we show that $P\in\boldsymbol{D}_{\varrho_1\rightarrow\varrho_2}$.
For any fixed $\tau_0$ and $\tau_1$ with $0<\tau_0<\tau_1$,
we have
\begin{align}
\label{eq:est-a-alpha}
&\norm{a_\alpha}_{\varrho_2,\tau_1+\tau'(\tau_1)}
\\
\nonumber
&\leq
\sum_{\beta\leq\alpha}
\frac{\norm{z^{\alpha-\beta}}_{\varrho_2,\tau_1}\norm{Fz^\beta}_{\varrho_2,\tau'(\tau_1)}}%
{(\alpha-\beta)!\beta!}
\\
\nonumber
&\leq
\sum_{\beta\leq\alpha}
\frac{\norm{z^{\alpha-\beta}}_{\varrho_1,\tau_1}C(\tau_1)\norm{z^\beta}_{\rho_1,\tau_1}}%
{(\alpha-\beta)!\beta!}
\\
\nonumber
&\leq
\sum_{\beta\le\alpha}
\frac{C(\tau_0,\tau_1)\tau_0^{-\abs{\alpha-\beta}/\rho_1}G_{\varrho_1,\abs{\alpha-\beta}}
C(\tau_1)C(\tau_0,\tau_1)\tau_0^{-\abs{\beta}/\rho_1}G_{\varrho_1,\abs{\beta}}}%
{(\alpha-\beta)!\beta!}
\\
\nonumber
&\leq
C(\tau_1)C(\tau_0,\tau_1)^2
\sum_{\beta\le\alpha}\binom{\alpha}{\beta}
\frac{G_{\varrho_1,\abs{\alpha}}}{\alpha!}
\tau_0^{-\abs{\alpha}/\rho_1}
\\
\nonumber
&=
C(\tau_1)C(\tau_0,\tau_1)^2
\cdot\frac{G_{\varrho_1,\abs{\alpha}}}{\alpha!}
2^{\abs{\alpha}}\tau_0^{-\abs{\alpha}/\rho_1}.
\end{align}
Here we used
\eqref{eq:norm-submultiplicative} at the first inequality,
\eqref{eq:norm-varrho-2-leq-1} and \eqref{conti-est1} at the second,
Lemma~\ref{normDUnMonome2} with $0<\tau_0<\tau_1$ at the third,
and Lemma~\ref{lem:Gq} at the fourth.

For a given $\eps>0$, we can take $\tau_0>0$ large enough
such that
\begin{equation*}
2\tau_0^{-1/\rho_1}<\eps.
\end{equation*}
Then, by choosing $\tau_1$ as $\tau_1>\tau_0$
and by putting $\sigma:=\tau+\tau'(\tau_1)$, we have
\begin{equation*}
\norm{a_\alpha}_{\varrho_2,\sigma}
\leq
C'
\cdot\frac{G_{\varrho_1,\abs{\alpha}}}{\alpha!}
\eps^{\abs{\alpha}}
\end{equation*}
for any $\alpha$,
which implies $P\in\boldsymbol{D}_{\varrho_1\rightarrow\varrho_2}$.

Now we show that $Pf=Ff$ for any $f\in A_{\varrho_1}$.
Since \eqref{a_alpha-def} corresponds to \eqref{eq:a-alpha-by-b-beta},
it follows from Proposition~\ref{prop:linear-poly-fps} that
this equality holds for any polynomial $f$.
It extends to $A_{\varrho_1}$ by the continuity
since polynomials form a dense subset of $A_{\varrho_1}$,
which was shown in Proposition~\ref{prop:taylor-exp-converge}.
\end{proof}

As a corollary, we give
\begin{cor}[{\cite[Corollary 6.5]{J}}]\label{cor:main1}
Let $\varrho$ be a proximate order for an order $\rho>0$.

\noindent{\rm (i)}
Suppose that $P\in\boldsymbol{D}_{\varrho}$ has the form \eqref{inf-ord}.
For an entire function $f\in A_{\varrho}$,
\[
Pf:=\sum_\alpha a_\alpha(z)\partial_z^\alpha f
\]
converges and $Pf\in A_{\varrho}$.
Moreover, $f\mapsto Pf$ defines a continuous endomorphism of $A_{\varrho}$.

\noindent{\rm (ii)}
Let $F$ be a continuous endomorphism of $A_{\varrho}$.
Then there is a unique $P\in\boldsymbol{D}_{\varrho}$
such that $Ff=Pf$ holds for any $f\in A_{\varrho}$.
\end{cor}

For the case of minimal type, we will prove:
\begin{thm}\label{thm:main2}
Let $\rho_i$ $(i=1,2)$ be two proximate orders for orders $\rho_i>0$
satisfying \eqref{eq:varrho1-is-O-varrho2}.

\noindent{\rm (i)}
Let $P\in\boldsymbol{D}_{\varrho_1\rightarrow\varrho_2,0}$
be of the form \eqref{inf-ord}.
For an entire function $f\in A_{\varrho_1,+0}$,
\[
Pf:=\sum_\alpha a_\alpha(z)\partial_z^\alpha f
\]
converges and $Pf\in A_{\varrho_2,+0}$.
Moreover, $f\mapsto Pf$ defines a continuous homomorphism
$P:A_{\varrho_1,+0}\longrightarrow A_{\varrho_2,+0}$.

\noindent{\rm (ii)}
Let $F:A_{\varrho_1,+0}\longrightarrow A_{\varrho_2,+0}$
be a continuous homomorphism.
Then there is a unique
$P\in\boldsymbol{D}_{\varrho_1\rightarrow\varrho_2,0}$
such that $Ff=Pf$ holds for any $f\in A_{\varrho_1,+0}$.
\end{thm}
\begin{proof}
Again we can make a replacement of proximate orders
as in Remark~\ref{rem:replace-prox-orders},
and we may assume from the beginning
that $\varrho_i$ ($i=1,2$) are normalized proximate orders
satisfying \eqref{eq:norm-varrho-2-leq-1}.

(i)
We assume condition \eqref{eq:est-D-rho1-rho2-minimal}
for $(a_{\alpha}(z))_{\alpha\in\NN^n}$, and denote by $\lambda_{\sigma}$
the constant $\lambda$ given there according to $\sigma$.
Similar computations as in \eqref{Pfnorm} yield
\begin{align}
\label{Pfnorm0}
&\sum_{\alpha}
\norm{a_{\alpha}\partial_z^{\alpha}f}_{\varrho_2,\sigma+\kappa\tau}
\\
\nonumber
&\leq
\sum_{\alpha}\norm{a_\alpha}_{\varrho_2,\sigma}
\norm{\partial_z^\alpha f}_{\varrho_2,\kappa\tau}
\\
\nonumber
&\leq
\sum_{\alpha}\norm{a_\alpha}_{\varrho_2,\sigma}
\norm{\partial_z^\alpha f}_{\varrho_1,\kappa\tau}
\\
\nonumber
&\leq
\sum_{\alpha}
C(\sigma)\frac{G_{\varrho_1,\abs{\alpha}}}{\alpha!}
\lambda_{\sigma}^{\abs{\alpha}}
\cdot C'(\tau)\norm{f}_{\varrho_1,\tau}
\frac{\alpha!}{G_{\varrho_1,\abs{\alpha}}}
\bigl(2\kappa n^{\rho_1/2}\tau\bigr)^{\abs{\alpha}/\rho_1}
\\
\nonumber
&=
C(\sigma)C'(\tau)\norm{f}_{\varrho_1,\tau}
\sum_{q\in\NN}\sum_{\abs{\alpha}=q}
\lambda_{\sigma}^q
\bigl(2\kappa n^{\rho_1/2}\tau\bigr)^{q/\rho_1}
\\
\nonumber
&=
C(\sigma)C'(\tau)
\norm{f}_{\varrho_1,\tau}
\sum_{q\in\NN}
{}_{n}H_q
\bigl(\sqrt{n}(2\kappa\tau)^{1/\rho_1}\lambda_{\sigma}\bigr)^q
\end{align}
for any $\sigma$, $\tau>0$, $f\in A_{\varrho_1,+0}$.
Here the constant $\lambda(\sigma)$ is given by \eqref{eq:est-D-rho1-rho2-minimal},
$\kappa=\kappa(\rho_1)$ depends only on $\rho_1$,
and $C(\sigma)$, $C'(\tau)$ are independent of $f$.
Note that in view of \eqref{eq:multi-choose},
the last sum is finite
if $\sqrt{n}(2\kappa\tau)^{1/\rho_1}\lambda_{\sigma}<1$.

For any $\eps>0$, we can choose $\sigma=\sigma(\eps)>0$
and $\tau=\tau(\eps)>0$ so that $\sigma+\kappa\tau\leq\eps$
and $\sqrt{n}(2\kappa\tau)^{1/\rho_1}\lambda_{\sigma}<1$.
In fact, we may first choose $\sigma$ as $0<\sigma<\eps/2$, which
determines $\lambda_{\sigma}$,
and then secondly choose $\tau>0$ as
\begin{equation*}
\kappa\tau<\eps/2,
\quad
\sqrt{n}(2\kappa\tau)^{1/\rho_1}\lambda_{\sigma}<1.
\end{equation*}
Therefore, for any $\varepsilon>0$, there exist
$C''(\eps)>0$ and $\tau(\eps)>0$ such that
\begin{equation*}
\norm{Pf}_{\varrho_2,\varepsilon}
\leq C''(\eps)\norm{f}_{\varrho_1,\tau(\eps)},
\quad\text{for $f\in A_{\varrho,+0}$},
\end{equation*}
which implies
that $P:A_{\varrho_1,+0}\rightarrow A_{\varrho_2,+0}$ is continuous.

(ii)
For a given continuous homomorphism
$F:A_{\varrho_1,+0}\rightarrow A_{\varrho_2,+0}$,
we construct $P=\sum_\alpha a_\alpha(z)\partial_z^\alpha$
via \eqref{a_alpha-def}
in the same way as in the proof of Theorem~\ref{thm:main1}(ii).

Let us show the convergence of \eqref{a_alpha-def} in $A_{\varrho_2,+0}$
together with the estimates.
For any $\sigma>0$, there exists $\tau_1>0$ and $C$
by continuity
such that
\begin{equation}\label{eq:conti-est-minimal}
\norm{F f}_{\varrho_2,\sigma/2}\leq C\norm{f}_{\varrho_1,\tau_1},
\quad\text{for $f\in A_{\varrho_1,+0}$}.
\end{equation}
Take $\tau_0$ with $0<\tau_0<\min\{\tau_1,\sigma/2\}$.
Similarly to \eqref{eq:est-a-alpha}, we have
\begin{align}
\label{eq:est0-a-alpha}
\norm{a_\alpha}_{\varrho_2,\sigma}
&\leq
\sum_{\beta\leq\alpha}
\frac{\norm{z^{\alpha-\beta}}_{\varrho_2,\sigma/2}\norm{Fz^\beta}_{\varrho_2,\sigma/2}}%
{(\alpha-\beta)!\beta!}
\\
\nonumber
&\leq
\sum_{\beta\leq\alpha}
\frac{\norm{z^{\alpha-\beta}}_{\varrho_1,\sigma/2}
C\norm{z^\beta}_{\varrho_1,\tau_1}}%
{(\alpha-\beta)!\beta!}
\\
\nonumber
&\leq
C
\sum_{\beta\le\alpha}
\frac{C(\tau_0,\sigma/2)\tau_0^{-\abs{\alpha-\beta}/\rho_1}G_{\varrho_1,\abs{\alpha-\beta}}
C(\tau_0,\tau_1)\tau_0^{-\abs{\beta}/\rho_1}G_{\varrho_1,\abs{\beta}}}%
{(\alpha-\beta)!\beta!}
\\
\nonumber
&\leq
CC(\tau_0,\sigma/2)C(\tau_0,\tau_1)
\sum_{\beta\le\alpha}\binom{\alpha}{\beta}
\frac{G_{\varrho_1,\abs{\alpha}}}{\alpha!}
\tau_0^{-\abs{\alpha}/\rho_1}
\\
\nonumber
&=
C C(\tau_0,\sigma/2)C(\tau_0,\tau_1)
\cdot\frac{G_{\varrho_1,\abs{\alpha}}}{\alpha!}
2^{\abs{\alpha}}\tau_0^{-\abs{\alpha}/\rho_1}.
\end{align}
Here we used
\eqref{eq:norm-submultiplicative} at the first inequality,
\eqref{eq:norm-varrho-2-leq-1} and \eqref{eq:conti-est-minimal} at the second,
Lemma~\ref{normDUnMonome2} twice with $0<\tau_0<\sigma/2$, $0<\tau_0<\tau_1$ at the third,
and Lemma~\ref{lem:Gq} at the fourth.

Therefore, by defining
\begin{equation*}
\lambda:=2\tau_0^{-1/\rho_1},
\end{equation*}
we have
\begin{equation*}
\norm{a_\alpha}_{\varrho,\sigma}
\leq
C'
\cdot\frac{G_{\varrho_1,\abs{\alpha}}}{\alpha!}
\lambda^{\abs{\alpha}}
\end{equation*}
for any $\alpha$.
Since $\sigma>0$ can be chosen arbitrarily,
this implies $P\in\boldsymbol{D}_{\varrho_1\rightarrow\varrho_2,0}$.

The remaining thing is to show $Ff=Pf$ for $f\in A_{\varrho_1,+0}$.
This can be done completely in the same way as in the case of normal type,
that is, the equality for polynomial $f$ due to Proposition~\ref{prop:linear-poly-fps}
can be extended to $A_{\varrho_1,+0}$ by continuity,
since polynomials form a dense subset of $A_{\varrho_1,+0}$,
for which we again refer to Proposition~\ref{prop:taylor-exp-converge}.
\end{proof}

Again, as a corollary, we give
\begin{cor}[{\cite[Corollary 6.6]{J}}]\label{cor:main2}
Let $\varrho$ be a proximate order for an order $\rho>0$.

\noindent{\rm (i)}
Let $P\in\boldsymbol{D}_{\varrho,0}$ be of the form \eqref{inf-ord}.
For an entire function $f\in A_{\varrho,+0}$, 
\[
Pf:=\sum_\alpha a_\alpha(z)\partial_z^\alpha f
\]
converges and $Pf\in A_{\varrho,+0}$.
Moreover, $f\mapsto Pf$ defines a continuous endomorphism
of $A_{\varrho,+0}$.

\noindent{\rm (ii)}
Let $F$ be a continuous endomorphism of $A_{\varrho,+0}$.
Then there is a unique 
$P\in\boldsymbol{D}_{\varrho,0}$ such that $Ff=Pf$ holds for any $f\in A_{\varrho,+0}$. 
\end{cor}
\begin{ex}[{\cite[Example 5.3]{aiosu}}]
We consider the following operator
\begin{equation*}
P=\exp\Bigl(-i\Bigl(t z+\frac{t^3}{6}\Bigr)\Bigr)
\exp\Bigl(\frac{i t}{2}\frac{\partial^2}{\partial z^2}
+\frac{t^2}{2}\frac{\partial}{\partial z}\Bigr)
\end{equation*}
in one variable $z\in\CC$ with parameter $t\in\CC$,
which is expected to be a solution operator
of the initial value problem for an unknown $\psi(t,z)$:
\begin{equation*}
\left\{
\begin{array}{l}
\displaystyle i\partial_t\psi
=\Bigl(-\frac{1}{2}\frac{\partial^2}{\partial z^2}+z\Bigr)\psi, \\
\psi(0,z)=\phi(z).
\end{array}
\right.
\end{equation*}
That is, $\psi$ is given by $\psi(t,z)=P\phi(z)$.
In \cite[Example 5.3]{aiosu}, we claimed that for each fixed $t$,
$P$ belongs to $\boldsymbol{D}_{p}$ in one variable $z$
if $1\leq p<2$.

For a (constant) order $p$, we can easily see
that $P$ for fixed $t$ does not belong to $\boldsymbol{D}_{p}$ in case $p\geq 2$,
and also that it belongs to $\boldsymbol{D}_{p,0}$ if and only if $1<p\leq 2$.
For example,
in order to prove $P\not\in\boldsymbol{D}_{p}$ for $p\geq 2$,
it suffices to show that for some $\eps>0$, there is no $C$ such that
\begin{equation*}
\frac{1}{j!}\leq C\frac{(2j)^{2j/p}}{(2j)!}\eps^{2j},
\quad\text{for any $j$}.
\end{equation*}
In particular, $P$ belongs to $\boldsymbol{D}_{2,0}$ but not to $\boldsymbol{D}_{2}$.

Now we consider a proximate order $\varrho(r)$
for the order $\rho=\lim_{r\rightarrow+\infty}\varrho(r)=2$.
If moreover $\varrho(r)$ satisfies
\begin{equation}\label{eq:rho-smaller-than-2}
\lim_{r\rightarrow+\infty}r^{\varrho(r)}/r^2=0,
\end{equation}
or, equivalently,
\begin{equation}\label{eq:t-over-phi-t-square}
\lim_{t\rightarrow+\infty}t/\varphi(t)^2=0,
\end{equation}
then, we see below that $P$ for fixed $t$ belongs
to $\boldsymbol{D}_{\varrho}$.

For the proof, it suffices to show that
\begin{equation*}
\exp\Bigl(\frac{i t}{2}\frac{\partial^2}{\partial z^2}\Bigr)
=\sum_{j\in\NN}\frac{1}{j!}
\Bigl(\frac{i t}{2}\Bigr)^j\Bigl(\frac{\partial}{\partial z}\Bigr)^{2j}
\in\boldsymbol{D}_{\varrho},
\end{equation*}
or, equivalently,
\begin{equation*}
\sup_j
\frac{\abs{t}^j}{2^j}\cdot\frac{1}{j!}
\cdot\frac{(2j)!}{G_{\hat{\varrho},2j}}\cdot\frac{1}{\eps^{2j}}
<+\infty,
\end{equation*}
for any $\eps>0$.
We have that
\begin{equation*}
\frac{\abs{t}^j}{2^j j!}\cdot\frac{(2j)!}{G_{\hat{\varrho},2j}\cdot\eps^{2j}}
=\frac{\abs{t}^j}{(2j)!!}\cdot
\frac{(2j)!(2\e)^{2j/\rho}}{\varphi(2j)^{2j}\eps^{2j}}
=
\Bigl(\frac{2j}{\varphi(2j)^2}
\cdot\frac{2\e\abs{t}}{\eps^2}\Bigr)^j
\cdot\frac{(2j-1)!!}{(2j)^j}
\end{equation*}
and that this sequence in $j$ converges to $0$ as $j\rightarrow\infty$.
Here we used \eqref{eq:t-over-phi-t-square} and $(2j-1)!!\leq(2j)^j$.
Therefore, it is in particular bounded.

Note also that there are many proximate orders for the order $\rho=2$
satisfying \eqref{eq:rho-smaller-than-2}.
Consider, for example, proximate orders $\varrho(r)$
given by $\varrho(r)=2+k\frac{\ln\ln r}{\ln r}$
or by $\varrho(r)=2+k\frac{\ln\ln\ln r}{\ln r}$ for large $r$
with a negative constant $k$.
\end{ex}


\end{document}